\documentclass[12pt]{amsart} 

\usepackage{amssymb,amsmath,latexsym,amsthm}
\usepackage[english]{babel}

\newtheorem{defi}{Definition}[section]
\newtheorem{proposition}[defi]{Proposition}
\newtheorem{theorem}[defi]{Theorem}
\newtheorem{conjecture}[defi]{Conjecture}
\newtheorem{remark}[defi]{Remark}
\newtheorem{corollary}[defi]{Corollary}
\newtheorem{lemma}[defi]{Lemma}
\newtheorem{example}[defi]{Example}

\newtheorem*{notation}{Notation}

\newcommand{\N}{{\mathbb{N}}}
\newcommand{\Z}{{\mathbb{Z}}}

\newcommand{\Fq}{{\mathbb{F}_q}}
\newcommand{\kt}{{K\{\tau\}}}
\newcommand{\gl}{{\mathrm{GL}}}

\newcommand{\lan}{{\Lambda_n}}
\newcommand{\rad}{{\mathrm{Rad}}}
\newcommand{\mor}{{\mathrm{Mor}}}
\newcommand{\Hom}{{\mathrm{Hom}}}
\newcommand{\End}{{\mathrm{End}}}
\newcommand{\Id}{{\mathrm{Id}}}
\newcommand{\jac}{{\mathrm{Jac}}}
\newcommand{\pic}{{\mathrm{Pic}}}
\newcommand{\tor}{{\mathrm{Tor}}}
\newcommand{\Card}{{\mathrm{Card}}}
\newcommand{\rank}{{\mathrm{rank}}}

\begin{document}

\title{$q$-Varieties and Drinfeld modules}

\author{Alain Thi\'ery}

\address{Alain Thi\'ery, Institut de Math\'ematiques de Bordeaux,
351 cours de la lib\'eration,
33400 Talence, France}

\email{alain.thiery@math.u-bordeaux.fr}

%\date{\today}

\begin{abstract}
Let $\Fq$ be the finite field with $q$ elements, $K$ be an algebraically closed field 
containing $\Fq$,  $\kt$ be the Ore ring of $\Fq$-linear
polynomials and $\lan$ be a free $\kt$-module of rank $n$.

In a first part, we prove that there is 
a bijection between the set of Zariski closed subsets of $K^n$
which are also $\Fq$-vector spaces, the so-called $q$-varities, and the set of radical
$\kt$-submodules of $\lan$. We also study the dimension of $q$-varieties and their tangent
spaces.

Let $F$ be a $q$-variety, $K\{F\} := \mor(F,K)$ be the set of $\Fq$-linear polynomial maps
from $F$ to $K$. Let $A=\Fq[T]$ and choose $\delta : A \longrightarrow K$ a ring morphism. 
 By definition, an $A$-module structure on $F$ is a ring morphism
$\Phi : A \longrightarrow \End(F)$ such that, for all $a\in A$,
$$d(\Phi_a) = \delta(a)\Id_{T(F)}$$
where $T(F)$ is the tangent space of $F$ and $d(\Phi_a)$ the differential map.
We prove that $K(F) := K(T)\otimes_{K[T]}K\{F\}$ has finite dimension over $K(T)$. 
This dimension is called the rank
of the $A$-module and is denoted by $r(F)$.

We then prove that there exists $c \in A\setminus \{0\}$ such that
for all $a\in A$, prime to $c$,
$$\tor(a,F) := \{x\in F \mid \Phi_a(x) = 0\} = (A/aA)^{r(F)}.$$
\end{abstract}

\maketitle

\section{Introduction}

In his seminal paper \cite{drinfeld}, V.G. Drinfeld defined what is now
called a Drinfeld module. Roughly speaking, it is an action of $A = \Fq[T]$ 
on an algebraically closed field $K$. More precisely, it is a ring morphism
$\Phi : A \longrightarrow \kt$, the Ore ring, such that the first coefficient of
$\Phi_a$ is $a$.

In another important paper \cite{anderson}, G. Anderson defined $T$-modules, which are a generalization
of Drinfeld modules. It is an action of
$A$ on $K^n$ such that the differential of the action of $a$ is just $a\Id_{K^n}+N$ where $N$ is nilpotent. 

In the present paper, we first study subvarieties of $K^n$ which are defined by $\Fq$-linear polynomial
equations. We call them $q$-varieties. The motivation is that we believe that 
$q$-varieties are the natural objects 
on which an action of $A$ can be defined.

In the first paragraphs, we prove a sort of Nullstellenstatz for $q$-varieties. We also define the notions
of morphism, irreducibility, dimension, tangent space for $q$-varieties. There is an obvious analogy with
the classical algebraic geometry, see chapter $1$ of \cite{hartshorne} for example.

Since they are additive algebraic groups, it is well-known, and easy to prove, 
that $q$-varieties are isomorphic
to some $K^r\times F$ where $F$ is a finite $\Fq$-vector space. So the reader can think that these objects are
not really worth studying, but $K$-vector spaces of finite dimension are all isomorphic to some $K^r$ and we
study them in full generality.

In paragraph $6$, we define the $A$-module structure in this context: let $F$ be a $q$-variety, an $A$-module 
structure is a morphism of 
$\Fq$-algebras
$\Phi : A \longrightarrow \End(F)$
such that, for all $a \in A$,
$$d(\Phi_a) = a\Id_{T(F)}$$
where $d(\Phi_a)$ is the differential of $\Phi_a$ and $T(F)$ the tangent space of $F$
(we forget the nilpotent part for simplicity).
Let $K\{F\} := \mor(F,K)$ be the set of $\Fq$-linear polynomial maps
from $F$ to $K$. Then $K\{F\}$ is a $K$-vector space and an $A$-module, so 
it is a $K\otimes_\Fq A =K[T]$-module. 

In \cite{anderson}, the rank of the module is by definition the
rank of $K\{F\}$ as a $K[T]$-module, if free and finitely generated. In this case, 
the module is said to be abelian. But, of course, all modules are not abelian. For example,
the trivial module, i.e. $\phi_a(x) = ax$, is not abelian.

To solve this difficulty, we prove that $K(F) := K(T)\otimes_{K[T]}K\{F\}$ has finite dimension over $K(T)$. 
This dimension is called the rank of the $A$-module and is denoted by $r(F)$. Obviously, in the case of 
abelian modules, our definition of the rank matches Anderson's definition. It is also easy to see
that the trivial module has rank equal to $0$.

In paragraph $7$, we study torsion points and prove the following result:
There exists $c \in A\setminus \{0\}$ such that
for all $a\in A$, prime to $c$,
$$\tor(a,F) := \{x\in F \mid \Phi_a(x) = 0\} = (A/aA)^{r(F)}.$$
This means that $A$-modules are almost regular.

In paragraph $8$, we define an analogue of the Jacobian and the Picard group. In paragraph $9$,
some analogues of Faltings'theorem, Mordell--Lang conjecture or Manin-Mumford conjecture are stated in
this context following L. Denis ideas.

\section{Definitions and first properties}

Let $p$ be a prime number and $K$ an algebraically closed field of characteristic $p$.
Fix $q = p^l$ a power of $p$. The finite subfield of $K$ with $q$ elements will be 
denoted by $\Fq$. 

Let $\tau = X^q$ be the Frobenius polynomial. Note that, for $i \geq 0$, 
$\tau^i := \tau\circ\ldots \circ\tau = X^{q^i}$. The set of $\Fq$-linear polynomials 
$\kt := 
\{P(X) = \sum_{i=0}^d a_iX^{q^i}\mid a_i \in K\}$ is a non-commutative 
unitary ring under composition. We will write $P\{\tau\}$ for $P(X)$.

It is well-know that $\kt$ is left and right euclidean 
(see \cite{goss}, Prop.1.6.2 and Prop.1.6.5). It implies the following lemma
(see \cite{goss}, Prop.5.4.8) :
\begin{lemma}\label{diag} Let $L$ be a $m\times n$ matrix with coefficients in $\kt$,
then there exist matrices $U\in \gl_m(\kt)$ and $V\in \gl_n(\kt)$ such that
$ULV$ is diagonal.
\end{lemma}

Let $n$ be a positive integer, $X_1,\ldots,X_n$ be $n$ indeterminates, 
$\tau_i = X_i^q$ ($1\leq i \leq n$) and 
$\lan := K\{\tau_1\}\oplus\ldots\oplus K\{\tau_n\}$, which is a free $\kt$-module 
of rank $n$ for the obvious action : $P\{\tau\}.Q(X_i) = P(Q(X_i))$.

\begin{defi}\label{qvar}
Let $S\subset\lan$, we define $Z(S)$, the ``zeroes of $S$'', by
$$Z(S) =\{(x_1,\ldots,x_n)\in K^n \mid\forall f\in S,\ f(x_1,\ldots,x_n) = 0\}.$$
The set $Z(S)$ is called a $q$-variety.
\end{defi}

\begin{remark}  The $q$-varieties are Zariski closed subsets
of $K^n$ and are also $\Fq$-vector spaces. It can be proved that it is a caracterisation
of $q$-varieties but we won't use this result here.
\end {remark}

\begin{proposition}\label{inter}
An intersection of $q$-varieties is a $q$-variety.
\end{proposition}

\begin{proof}
We have trivially $\bigcap_{i\in I}Z(S_i) = Z(\cup_{i\in I}S_i)$.
\end{proof}

\begin{defi}\label{adher} Let $F\subset K^n$, 
we define $\overline{F}$ as the intersection of all 
$q$-varieties containing $F$.
Using proposition \ref{inter}, $\overline{F}$ is the smallest $q$-variety containing $F$.
\end{defi}

\begin{remark}\label{finite}
Let $\Lambda$ be the $\kt$-module generated by $S$ in $\lan$. One gets immediatly that 
$Z(S) = Z(\Lambda)$, so, in the definition of a $q$-variety, we can suppose that $S$
is a $\kt$-submodule of $\lan$. Furthermore, since $\kt$ is Noetherian, the 
$\kt$-submodules of $\lan$ are finitely generated, so that a $q$-variety can be defined
by a finite number of equations.
\end{remark}

\begin{defi}\label{module} Let $F\subset K^n$, we define $M(F)$ by
$$M(F) =\{f\in \lan\mid \forall (x_1,\ldots,x_n)\in F,\ f(x_1,\ldots,x_n) = 0\}.$$
It is immediat that $M(F)$ is a $\kt$-submodule of $\lan$.
\end{defi}

\begin{proposition}\label{nullpart}
 Let $F\subset K^n$. We have the following equality :
$$Z(M(F))= \overline{F}.$$
\end{proposition}

\begin{proof} By definitions \ref{qvar} and \ref{module}, $F\subset Z(M(F))$. 
Definition \ref{adher} implies that $\overline{F}\subset  Z(M(F))$. Conversely, 
$\overline{F}$ is a $q$-variety, so there exists a submodule $M$ such that 
$\overline{F}= Z(M)$. Since $F\subset \overline{F}$, we immediatly have that
$M\subset M(F)$. It implies trivially that $Z(M(F)) \subset Z(M) =  \overline{F}$. 
\end{proof}

\begin{remark}
The module $M(F)$ has the following property : let $f\in \lan$ be such that
$\tau f \in M(F)$, then $f\in M(F)$. It means that the quotient module $\lan/M(F)$ has no
$\tau$-torsion.
\end{remark}

This leads to the following definition :

\begin{defi}\label{rad}
Let $\Lambda$ be a submodule of $\lan$. We define $\rad(\Lambda)$ by
$$\rad(\Lambda) = \{ f \in \lan \mid \exists N\in \N,\ \tau^N f \in \Lambda\}.$$
A module $\Lambda$ such that $\rad(\Lambda) = \Lambda$ is said to be radical. 
For example, for any
$F\subset K^n$, the module $M(F)$ is radical.
\end{defi}

\begin{proposition}\label{radismod}
Let $\Lambda$ be a submodule of $\lan$, then $\rad(\Lambda)$ is also a submodule of $\lan$ 
and $\rad(\Lambda)$ is radical.
\end{proposition}
\begin{proof} We only have to prove that if $f\in \rad(\Lambda)$ and $P\in \kt$, 
then $Pf\in \rad(\Lambda)$.
By definition, there exists $N\in \N$ such that $\tau^Nf \in \Lambda$. 
Let $P =\sum_{i=0}^d a_i\tau^i$, then $\tau^N P =\sum_{i=0}^d a_i^{q^N}\tau^{N+i}=
(\sum_{i=0}^d a_i^{q^N}\tau^{i})\tau^N = Q\tau^N$ with
$Q=\sum_{i=0}^d a_i^{q^N}\tau^{i}\in \kt$. 
Now we have $\tau^N Pf= Q\tau^Nf$. But $\tau^Nf\in \Lambda$ and 
$\Lambda$ is a $\kt$-module, hence
 $\tau^N Pf\in \Lambda$. By definition, $Pf\in \rad(\Lambda)$.
\end{proof}

\begin{defi}\label{morph}
Let $F\subset K^n$ and $H\subset K^m$ be $q$-varieties. A morphism from 
$F$ to $H$ is a map $\psi : F \longrightarrow H$ such that there exists $f_1,\ldots,f_m
\in \lan$ satisfaying :
$$\forall x\in F,\ \psi(x) =(f_1(x),\ldots,f_m(x)).$$
An isomorphism is a bijective morphism $\psi$ such that $\psi^{-1}$ is also a morphism.
\end{defi}

\begin{example}
The map $\tau : K\longrightarrow K$ is a morphism. 
It is bijective but it is not an isomorphism.

Let $P\in \kt$ and  $\psi : K^2\longrightarrow K^2$ be the morphism defined by the matrix 
$
\begin{pmatrix}
\tau^0& P\\
0&\tau^0
\end{pmatrix}
$. It means that
$\psi(x_1,x_2) =(x_1+P(x_2),x_2)$.
It is clear that $\psi$
is an isomorphism since  $\psi^{-1}$ is given by the matrix
$
\begin{pmatrix}
\tau^0& -P\\
0&\tau^0
\end{pmatrix}
$.
\end{example}

\begin{theorem}\label{fonct}
Let $F\subset K^n$ and $H\subset K^m$ be $q$-varieties and $\mor(F,H)$ be the $\Fq$-vector 
space of all morphisms from $F$ to $H$, then there exists a fonctorial isomorphism of
$\Fq$-vector spaces
$$\mor(F,H) \simeq \Hom_{\kt}(\Lambda_m/M(H),\lan/M(F)).$$
\end{theorem}

\begin{proof} Let $\psi : F \longrightarrow H$ be a morphism given by 
$\psi(x) =(f_1(x),\ldots,f_m(x))$. We can define a $\kt$-linear map $u_\psi$ from $\Lambda_m$ to
$\lan/M(F)$ by $u_\psi(X_j) = f_j \mod M(F)$. It is clear that it does not depend on the choice
of the $f_j$. Let $g = \sum_{j=1}^m g_j(X_j) \in M(H)$ with $g_j\in \kt$ and $x\in F$,
\begin{align*}
u_\psi(g)(x) &= u_\psi\left(\sum_{j=1}^m g_j(X_j)\right)(x)\\
&=\sum_{j=1}^m g_j\left(u_\psi(X_j)\right)(x)\\
&=\sum_{j=1}^m g_j(f_j)(x)\\
&=g(f_1(x),\ldots,f_m(x))\\
&=0
\end{align*}
since $(f_1(x),\ldots,f_m(x))\in H$ and $g\in M(H)$. It follows that $u_\psi(g) = 0 \mod M(F)$
so $u_\psi : \Lambda_m/M(H) \longrightarrow \lan/M(F)$ is well-defined. 

Conversely, let $u : \Lambda_m/M(H) \longrightarrow \lan/M(F)$ be a $\kt$-modules morphism
and $f_1,\ldots,f_m\in\lan$ be such that
$$u(\overline{X_j}) \equiv \overline{f_j} \mod M(F).$$
Let us define $\psi_u: F \longrightarrow K^m$ by
$$\forall x\in F,\ \psi_u(x) =(f_1(x),\ldots,f_m(x)).$$
It is clear that $\psi_u$ does not depend on the choice of the $f_j$. We now have to show
that $\psi_u(F)$ is included in $H$. Let $g = \sum_{j=1}^m g_j(X_j) \in M(H)$ and $x\in F$,
\begin{align*}
g\left(\psi_u(x)\right) &=\sum_{j=1}^m g_j(f_j)(x) \\
&=\sum_{j=1}^m g_j\left(u(\overline{X_j})\right)(x)\\
&=u\left(\sum_{j=1}^m g_j(\overline{X_j})\right)(x)
\text{ (since } u \text{ is } \kt\text{-linear)}\\
&=u\left(\overline{\sum_{j=1}^m g_j(X_j)}\right)(x)\\
&=u(\overline{g})(x) =u(0)(x) = 0.
\end{align*}
By definition, it implies that $\psi_u(x) \in Z(M(H))$. But $Z(M(H)) = \overline{H} = H$ by 
proposition \ref{nullpart}, proving that $\psi_u(F)\subset H$.

It is now straightforward to prove that $\tilde{u} : \psi \mapsto u_\psi$ 
and $\tilde{\psi} : u \mapsto \psi_u$ 
are reciprocal isomorphisms.
\end{proof}

The previous theorem implies that the $\kt$-module $\lan/M(F)$ 
depends only on the isomorphism class of $F$, so 
we can set the following definition :

\begin{defi}\label{fonction}
Let $F\subset K^n$ be a $q$-variety. The $\kt$-module $\lan/M(F)$ is called the module
of $\Fq$-linear functions on $F$ and is denoted by $K\{F\}$. By construction, $K\{F\}=\mor(F,K)$.
\end{defi}

\begin{proposition}
Let $F\subset K^n$ and $H\subset K^m$ be $q$-varieties, and $\psi : F \longrightarrow H$ 
be a morphism from $F$ to $H$. Then for any $q$-variety $G\subset H$, $\psi^{-1}(G)$ is
$q$-variety.
\end{proposition}

\begin{proof} Let $S\subset \Lambda_m$ be a set defining $H$ : $H = Z(S)$. One gets from 
definitions that $\psi^{-1}(G) = H\cap Z(\{g(f_1,\dots,f_m)\mid g \in S\})$ where 
$f_1,\dots,f_m$ are as in definition \ref{morph}.
\end{proof}

\begin{remark}\label{noyau}
The kernel of a morphism is a $q$-variety but, indeed, any $q$-variety can be expressed
as a kernel : let $F\subset K^n$ be a $q$-variety defined by a finite number of equations
$f_1,\ldots,f_m$ (see remark \ref{finite}). Then the morphism 
$\psi : K^n \longrightarrow K^m$ defined by $\psi(x) = (f_1(x),\ldots,f_m(x))$ has a kernel
equal to~$F$.
\end{remark}

\begin{lemma}\label{finitevar}
Let $F\subset K^n$ be a finite $\Fq$-vector space, then $F$ is a $q$-variety.
\end{lemma}

\begin{proof}
We prove it by induction on $d = \dim_\Fq F$. If $d = 0$, $F =\{0\} = Z(X_1,\ldots,X_n)$ is a
$q$-variety. If $d = 1$, $F = \Fq x$ for some $x\in K^n$. There exists a $K$-linear bijective
map $\psi$ such that $\psi(x) = (1,0,\ldots,0)$. But $\Fq\times\{0\}\times\ldots\times \{0\} =
Z(X_1^q-X_1,X_2,\ldots,X_n)$, so it is $q$-variety and, by construction,
$F = \psi^{-1}(\Fq\times\{0\}\times\ldots\times \{0\})$, so it is also a $q$-variety.

Suppose that any $\Fq$-vector space of dimension $d$ is a $q$-variety and let $F$ be an 
$\Fq$-vector space of dimension $d+1$. Choose $H\subset F$ a subvector space of dimension $1$
and $\psi$ a morphism such that $\ker \psi = H$ (see remark \ref{noyau}). 
Then $\dim_\Fq \psi(F) = d$, so it is a
$q$-variety. By construction $F = \psi^{-1}(\psi(F))$, so it is also a $q$-variety.
\end{proof}

\section{Main theorem on $q$-Varieties}

\begin{lemma}\label{null} Let $\Lambda$ be a submodule of $\lan$, then
$$M(Z(\Lambda)) = \rad(\Lambda).$$
\end{lemma}

\begin{proof} By definition, $\Lambda\subset M(Z(\Lambda))$. Taking radicals, 
we have
$\rad(\Lambda)\subset \rad(M(Z(\Lambda))) = M(Z(\Lambda))$ 
since $M(F)$ is a radical module for any $F$.

Conversely, let $f_1,\ldots,f_m \in \lan$ be a finite generating set for $\Lambda$. For
$1 \leq i \leq m$, write $f_i = \sum_{j=1}^n L_{i,j}(X_j)$ with $ L_{i,j}\in\kt$. 
The matrix $L = (L_{i,j})_{\substack{1 \leq i \leq m\\1 \leq j \leq n}}$ is such that 
$\ker L = Z(\Lambda)$. Applying Lemma \ref{diag}, 
there exist matrices $U\in \gl_m(\kt)$ and $V\in \gl_n(\kt)$ such that
$D = ULV$ is diagonal. Without loss of generality, we can suppose that $D_{ii}\not = 0$ for 
$i\leq r$ and $D_{ii} = 0$ if $i > r$. It means that, up to changing variables and the 
generating set, one can suppose that 
$$\Lambda = \kt P_1(X_1)+\ldots+\kt P_r(X_r)$$
with $P_1,\ldots,P_r \in \kt\setminus \{0\}$. It follows that 
$$Z(\Lambda) = \ker P_1\times\ldots\times\ker P_r\times K^{n-r}.$$
Let $g = \sum_{i=1}^n g_i(X_i) \in M(Z(\Lambda))$. For all $i > r$ and $x_i \in K$, 
$(0,\ldots,0,x_i,0,\ldots,0)\in Z(\Lambda)$. It implies that $g_i(x_i) = 0$, 
so that $g_i = 0$. For $1\leq i \leq r$, write $P_i = \tau^{N_i} Q_i$ with $Q_i$ separable.
By construction, $\ker Q_i = \ker P_i \subset \ker g_i$. Since $\kt$ is Euclidean, 
$g_i = SQ_i+R$ with $\deg_\tau R < \deg_\tau Q_i$. But $\ker Q_i \subset \ker R$ and 
$\dim_\Fq \ker Q_i = \deg_\tau Q_i > \deg_\tau R = \dim_\Fq \ker R$ unless $R$ is zero.
It follows that $g_i= SQ_i$ and 
$$\tau^{N_i}g_i= \tau^{N_i}SQ_i = T\tau^{N_i}Q_i = TP_i\in \Lambda$$
where $T\in \kt$ is such that $\tau^{N_i}S = T\tau^{N_i}$ 
(see proof of proposition \ref{radismod}). Taking $N = \max_{1\leq i\leq r} N_i$, 
one gets immediatly that $\tau^N g \in \Lambda$.
\end{proof}

We can now summarize this in the following theorem.

\begin{theorem}[Main Theorem]\label{bijec}
The map $Z$ from the set of radical modules of $\lan$ to the set of $q$-varieties 
included in $K^n$
and the map $M$ from 
the set of $q$-varieties included in $K^n$ to the set of radical modules 
of $\lan$ are reciprocal bijections.
\end{theorem}

We now study the direct image of a $q$-variety. It is possible to use general theorems on 
algebraic groups (see \cite{borel}) 
but we give a self-contained proof. First we need the following lemma :

\begin{lemma}\label{finiteindex}
Let $H\subset F\subset K^n$ be $\Fq$-vector spaces such that $F/H$ is finite and $H$ is 
a $q$-variety, then $F$ is also a $q$-variety.
\end{lemma}
\begin{proof} Let $\psi$ be a morphism such that $H = \ker \psi$ (see remark \ref{noyau}).
Then $\psi(F) \simeq F/H$ is finite, hence it is a $q$-variety by lemma \ref{finitevar}.
By construction, $F = \psi^{-1}(\psi(F))$, so it is a $q$-variety.
\end{proof}

\begin{theorem}\label{direct}
Let $F$ and $G$ be $q$-varieties, 
and $\psi : F\longrightarrow G$ be a morphism. Then
$\psi(F)$ is a $q$-variety.
\end{theorem}

\begin{proof}
In proof of lemma \ref{null}, we showed that there exists
$P_1,\ldots,P_r \in \kt\setminus \{0\}$ such that, up to an automorphism of $K^n$, 
$F = \ker P_1\times\ldots\times\ker P_r\times K^{n-r}$. Then 
$H = \{0\}\times\ldots\times \{0\}\times K^{n-r}$ is a $q$-variety such that
$F/H$ is finite. Hence $\psi(F)/\psi(H)$ is also finite. Applying lemma \ref{finiteindex}, 
it is sufficient to prove that $\psi(H)$ is a $q$-variety. Since $H\simeq K^{n-r}$, 
we can consider $\psi_{|H}$ as a morphism from $K^{n-r}$ to $K^m$. Using lemma \ref{diag},
$\psi(H)$ is equal, up to an automorphism of $K^m$, to the direct image of $K^{n-r}$ by
a diagonal morphism. Since for any non-zero polynomial $P\in \kt$, we have $P(K) = K$, 
the image is clearly of the form  
$K^s\times\{0\}\times\ldots\times\{0\}$, hence is a $q$-variety.
\end{proof}

We now give a few consequences of this theorem.

\begin{proposition}\label{somme}
Let $F_1\subset K^n$ and $F_2\subset K^n$ be $q$-varieties, then $F_1+F_2$ is also a
$q$-variety. Indeed, if $F_1= Z(\Lambda_1)$ and $F_2 =  Z(\Lambda_2)$ for some modules
$\Lambda_1$ and $\Lambda_2$, then
$F_1+F_2 = Z(\Lambda_1\cap \Lambda_2)\footnote{A proof, directly from the definition, 
would be welcome.}$.
\end{proposition}
\begin{proof}
It is clear that $F_1\times F_2$ is a $q$-variety in $K^n\times K^n = K^{2n}$ and that 
$\psi :  K^n\times K^n \longrightarrow K^n$ defined by $\psi(x,y) = x+y$ is a morphism. Hence 
by theorem \ref{direct}, $F_1+F_2 = \psi(F_1\times F_2)$ is a $q$-variety.

Now, let $\Lambda = M(F_1+F_2)$. Since $F_1\subset F_1+F_2$, $\Lambda \subset M(F_1) = 
\rad(\Lambda_1)$ by lemma \ref{null}. 
By the same argument, $\Lambda \subset \rad(\Lambda_2)$. It follows
that $\Lambda \subset \rad(\Lambda_1)\cap \rad(\Lambda_2) = \rad(\Lambda_1\cap \Lambda_2)$ 
(the last equality is left to the reader). Applying $Z$, we get
$Z(\Lambda_1\cap \Lambda_2) = Z(\rad(\Lambda_1\cap\Lambda_2))\subset Z(\Lambda) = 
\overline{F_1+F_2} = F_1+F_2$ since $F_1+F_2$ is a $q$-variety by theorem \ref{somme}.
Conversely, since $\Lambda_1\cap\Lambda_2 \subset \Lambda_1$, $F_1 = Z(\Lambda_1) 
\subset Z(\Lambda_1\cap \Lambda_2)$. Using the same argument with $F_2$, we have 
$F_2\subset Z(\Lambda_1\cap \Lambda_2)$, hence $F_1+F_2\subset Z(\Lambda_1\cap \Lambda_2)$.
\end{proof}

\begin{proposition}\label{quotient}
Let $H\subset F$ be $q$-varieties, 
then there exists a $q$-variety, denoted by $F/H$, and a 
morphism $\Pi : F \longrightarrow F/H$ with $\ker \Pi = H$  
satisfaying the following property :
for any morphism $\psi : F \longrightarrow G$ with $\psi_{|H} =0$, there exists a 
unique morphism $\overline{\psi} : F/H  \longrightarrow G$ such that $\psi = \overline{\psi}\circ \Pi$. 
The $q$-variety $F/H$ is called the quotient of $F$ by $H$ and the couple $(F/H, \Pi)$
is unique up to isomorphism. Furthermore, the map $\Pi$ is surjective.
\end{proposition}

\begin{proof}
It is clear that, if it exists, the quotient is unique. Let us prove the existence.

Let $f_1,\ldots,f_m \in \Lambda_n$ be a generating set of $M(H)$ and define 
$\Pi : K^n \longrightarrow K^m$ by $\Pi(x) =(f_1(x),\ldots,f_m(x))$. It is clear that
$\Pi$ is a morphism and that $\ker \Pi = H$. By theorem \ref{direct}, $\Pi(F)$ is a 
$q$-variety that will be denoted by $F/H$ 
(note that this $F/H$ is isomorphic to the standard one as an $\Fq$-vector space). 
We now write $\Pi$ for 
$\Pi_{|F} : F \longrightarrow F/H$ for simplicity. 

Let $\psi : F \longrightarrow G\subset K^r$ be a morphism 
with $\psi_{|H} =0$.
By definition, there exists $g_1,\ldots,g_r \in 
\Lambda_n$ such that for all $x\in F$, $\psi(x) = (g_1(x),\ldots,g_r(x))$.
The condition  $\psi_{|H} =0$ means that $g_i \in M(H)$ for all $1\leq i \leq r$. Hence there 
exists $a_{i,j} \in \kt$ such that $g_i = \sum_{j=1}^m a_{i,j}f_j$. Let us define the morphism
$\overline{\psi} :  K^m \longrightarrow K^r$ by 
$$\overline{\psi}(x_1,\ldots,x_m) = (\sum_{j=1}^m a_{1,j}(x_j),\ldots,\sum_{j=1}^m a_{r,j}(x_j)).$$
By construction, for all $x\in F$, we have $\psi(x) = \overline{\psi}(\Pi(x))$. It implies that 
$\overline{\psi}_{|F/H}$ is a morphism from $F/H$ to $G$, proving the existence of $\overline{\psi}$. 

Since $\Pi$ is surjective, the map $\overline{\psi}$ is unique.

\end{proof}

\begin{remark} \label{exactsequence} Let $H\subset F$ be $q$-varieties, we still denote 
$M(H) = \{f\in K\{F\} \mid \forall x \in H, f(x) = 0\}$. By construction $K\{H\} = K\{F\}/M(H)$
and $K\{F/H\}\simeq M(H)$.
In other words, the following 
sequence of $\kt$-modules is exact
$$0\rightarrow K\{F/H\} \rightarrow K\{F\}\rightarrow K\{H\}\rightarrow 0.$$
\end{remark}

\begin{remark}
Let $\psi : F \longrightarrow H$ be a morphism of $q$-varieties. The previous proposition shows
that $\psi$ induces a bijective morphism $\overline{\psi}$
from $F/\ker \psi$ to $\psi(F)$. Allthough this 
morphism is a bijection, it is not necessarily an isomorphism since the reciprocal bijection
might not be a morphism of $q$-varieties, take $\psi(x) = x^q$ for instance.
\end{remark}

This leads to the following definition :

\begin{defi}\label{separ}
Let $\psi : F \longrightarrow H$ be a morphism of $q$-varieties. We say that $\psi$ is separable if 
the bijective morphism $\overline{\psi} : F/\ker \psi \longrightarrow \psi(F)$ is an isomorphism.
\end{defi}

\section{Irreducible $q$-Varieties and dimension}

\begin{defi} \label{irred}
Let $F$ be a $q$-variety. It is said to be irreducible if the only
sub-$q$-variety of finite index is $F$ itself.
\end{defi}

\begin{example}
\begin{enumerate}
\item It is clear that $\{0\}$ is irreducible.
\item The sub-$q$-varieties of $K$ are finite or equal to $K$, hence $K$ is irreducible 
since $K$ is infinite.
\end{enumerate}
\end{example}

\begin{proposition}\label{sommeimageirred}
\begin{enumerate}
\item Let $F$ be an irreducible $q$-variety and $\psi : F \longrightarrow G$ be a morphism, 
then $\psi(F)$ is irreducible.
\item Let $F_1\subset K^n$ and $F_2\subset K^n$ be irreducible $q$-varieties, then $F_1+F_2$
is irreducible. It implies that $K^n$ is irreducible.
\end{enumerate}
\end{proposition} 

\begin{proof}
\begin{enumerate}
\item Let $H\subset \psi(F)$ be a $q$-variety such that $\psi(F)/H$ is finite. 
Since the induced map $\psi : F/\psi^{-1}(H) \longrightarrow \psi(F)/H$ is an isomorphism of
$\Fq$-vector spaces, $F/\psi^{-1}(H)$ is finite. Hence $\psi^{-1}(H) = F$, and $H = \psi(F)$.
\item Let $H\subset F_1+F_2$ be a $q$-variety such that $(F_1+F_2)/H$ is finite. Since the
canonical map $F_1/(F_1\cap H) \longrightarrow (F_1+F_2)/H$ is injective, $F_1/(F_1\cap H)$ is 
finite. Hence $F_1 = F_1\cap H$, that is $F_1\subset H$. By symmetry, 
we also have $F_2\subset H$. It follows that $F_1+F_2 \subset H$.
\end{enumerate}
\end{proof}

\begin{proposition}\label{sanstors}
Let $F$ be a $q$-variety and $K\{F\}$ be the $\kt$-module of $\Fq$-linear functions on $F$. 
Then the following 
properties are equivalent
\begin{enumerate}
\item $F$ is irreducible.
\item $F$ is isomorphic to $K^m$ for some $m$.
\item $K\{F\}$ is a free $\kt$-module.
\item $K\{F\}$ is a torsion free $\kt$-module.
\end{enumerate}
\end{proposition}

\begin{proof} We can suppose, as in proof 
of lemma \ref{null}, that 
$$M(F) = \kt P_1(X_1)+\ldots+\kt P_r(X_r)$$
for some $P_1,\ldots,P_r \in \kt\setminus \{0\}$. It follows that 
$$F = \ker P_1\times\ldots\times\ker P_r\times K^{n-r}.$$

Suppose that $F$ is irreducible. The $q$-variety 
$\{0\}\times\ldots\times\{0\}\times K^{n-r}$ is clearly of finite index in $F$, so it must be 
equal to $F$, hence $F = \{0\}\times\ldots\times\{0\}\times K^{n-r}$ is isomorphic to $K^{n-r}$.

Suppose that $F$ is isomorphic to $K^m$. Then, by theorem \ref{fonct}, 
$K\{F\}$ is isomorphic to $K\{K^m\} = \Lambda_m$. Hence
$K\{F\}$ is free.

Suppose that $K\{F\}$ is free, then, trivially, $K\{F\}$ is torsion free.

Suppose that $K\{F\} = \lan/M(F)$ is torsion free. For $1\leq i\leq r$, we have
$P_i\{\tau\}.X_i \equiv 0$ in $\lan/M(F)$, 
hence $X_i\in M(F)$. It implies immediatly that $P_i = X_i$
and $F =\{0\}\times\ldots\times\{0\}\times K^{n-r}$ which is isomorphic to the irreducible
$q$-variety $ K^{n-r}$.
\end{proof}

\begin{lemma}\label{compo}
Let $F$ be a
$q$-variety. Then there exists 
a necessarily unique irreducible sub-$q$-variety which is maximal for inclusion. 
It is denoted by $\mathring{F}$ and is called the irreducible component of $F$. 
Furthermore, $F/\mathring{F}$ is finite and $\mathring{F}$ is the only
irreducible $q$-variety satisfying such property.
\end{lemma}

\begin{proof} As in proof of proposition \ref{sanstors}, we can suppose, up to an automorphism 
of $K^n$, that $F = F_1\times\ldots\times F_r\times K^{n-r}$ with the $F_i$ being finite 
$\Fq$-vector spaces. Let $H = \{0\}\times\ldots\times\{0\}\times K^{n-r}$. It is an irreducible
sub-$q$-variety of finite index in $F$.

Now let $G\subset F$ be an irreducible $q$-variety. Since $F/H$ is finite, $G/(G\cap H)$ is 
also finite. But $G$ is irreducible, so $G = G\cap H$, proving that $G\subset H$.
\end{proof}

\begin{lemma}\label{imagecomp}
Let $F$ and $H$ be $q$-varieties and $\psi : F \longrightarrow H$ be a morphism. Then
$$\psi(\mathring{F}) = \mathring{\psi(F)}.$$
\end{lemma}

\begin{proof}
By proposition \ref{sommeimageirred}, $\psi(\mathring{F})$ is irreducible. Since $\mathring{F}$ has
finite index in $F$, $\psi(\mathring{F})$ has finite index in $\psi(F)$, proving the lemma.
\end{proof}

\begin{defi}
Let $F$ be a $q$-variety. If $F_0\subsetneq F_1 \subsetneq\ldots \subsetneq F_m \subset F$ 
is a chain of irreducible $q$-varieties, the integer $m$ is called the length of the chain.
\end{defi}

\begin{theorem}\label{dimen}
Let $F\subset K^n$ be a $q$-variety. Then,
\begin{enumerate}
\item All chains included in $F$ have length less than $n$.
\item All maximal chains included in $F$  have the same length.
\end{enumerate}
We will define $\dim F$ to be the maximal length of a chain included in $F$. For example,
we have $\dim K^n = n$.
\end{theorem}

\begin{proof} Without loss of generality, we can replace $F$ by $\mathring{F}$ and 
suppose that $F$ is irreducible. 
It follows from proposition \ref{sanstors} that $F$ is isomorphic to some $K^{m}$. So it is
sufficient to prove the theorem for $F = K^n$. We will do it by induction on $n$, 
including the property that $\dim K^n = n$. 
It is obviously true
for $n = 0$ and $n =1$ since the only irreducible are $\{0\}$ in the first case and 
$\{0\}$ and $K$ in the second one.

We suppose that the theorem is true up to $n$. Now let 
$F_0\subsetneq F_1 \subsetneq\ldots \subsetneq F_m \subset K^{n+1}$ 
be a chain of irreducible $q$-varieties. If $m = 0$, $m\leq n+1$.
If $m > 0$, $F_{m-1}$is an irreducible $q$-variety, so, up to an automorphism of $K^{n+1}$,
$F_{m-1} = \{0\}\times\ldots\times\{0\}\times K^{n+1-r}$. Since 
$F_{m-1} \not = K^{n+1}$, $r \not = 0$, hence we can apply the induction hypothesis to 
$F_{m-1} \simeq K^{n+1-r}$ :
$m-1 \leq n+1-r$. It implies $m \leq n+1$.

Suppose now that $F_0\subsetneq F_1 \subsetneq\ldots \subsetneq F_m \subset K^{n+1}$ is a maximal
chain (for example a chain of maximal length). 
Since $K^{n+1}$ is irreducible, we must have $F_m = K^{n+1}$. Up to an automorphism, 
one can suppose that $F_{m-1}=\{0\}\times\ldots\times\{0\}\times K^{n+1-r}$. The number of zeros
is the product must be equal to $1$, otherwise we could replace one of them by $K$
 to get an extra
irreducible in the chain which is supposed to be maximal. Hence $F_{m-1} = \{0\}\times K^n$.
But $F_0\subsetneq F_1 \subsetneq\ldots \subsetneq F_{m-1}$ is a maximal chain. By induction, 
its length is $\dim K^n = n$. So $m-1 = n$, that is $m = n+1$, proving that all maximal chains 
have the same length and that $\dim K^{n+1} = n+1$.
\end{proof}

\begin{corollary}\label{corodim} Let $F$ be a $q$-variety. Then
$$\dim F = \rank_{\kt} K\{F\}.$$
\end{corollary}

\begin{proof}
Let $\mathring{F}$ be the irreducible component of $F$. Hence $F/\mathring{F}$ is finite and 
 the $K$-vector space $K\{F/\mathring{F}\} \subset \hom_{\Fq}(F/\mathring{F},K)$ has finite dimension.
It follows that $\rank_{\kt} K\{F/\mathring{F}\} = 0.$ Now by remark \ref{exactsequence}, 
$$ \rank_{\kt} K\{F\} = \rank_{\kt} K\{\mathring{F}\} + \rank_{\kt} K\{F/\mathring{F}\}.$$
Furthermore, $\dim F = \dim \mathring{F}$. Hence without loss of generality, we can suppose that
$F$ is irreducible.

By proposition \ref{sanstors}, we can suppose that
$F = K^n$. But $K\{K^n\} = \Lambda_n$ and $\dim K^n = n$ by theorem \ref{dimen}.
\end{proof}

\begin{theorem}\label{rang}
Let $F$ and $H$ be $q$-varieties and $\psi : F \longrightarrow H$ be a morphism. Then
$$\dim F = \dim \ker \psi + \dim \psi(F).$$
\end{theorem}

\begin{remark}
\begin{enumerate}
\item Appliying the theorem to the canonical morphism $\Pi : F \longrightarrow F/H$ gives
$\dim F/H = \dim F -\dim H$ as expected.
\item It implies immediatly that $\dim F/\ker \psi = \dim \psi(F)$.
\end{enumerate}
\end{remark}

\begin{proof}\footnote{It is certainly possible to prove the formula using corollary \ref{corodim} 
and the rank of a module.}  
Let $r = \dim \ker \psi$, $s = \dim \psi(F)$, 
$\{0\} = K_0 \subsetneq K_1\subsetneq \ldots \subsetneq K_r \subset \ker \psi$ 
be a maximal chain of irreducibles and 
$\{0\} = I_0 \subsetneq I_1\subsetneq \ldots \subsetneq I_s \subset  \psi(F)$
be a maximal chain of irreducibles. For $0\leq i\leq s$, set $F_i = \mathring{\psi^{-1}(I_i)}$.
In  particular $F_0 = \mathring{\ker\psi} = K_r$. Furthermore, by lemma \ref{compo}, 
$F_i$ has finite index in $\psi^{-1}(I_i)$, hence $\psi(F_i)$ has finite index in 
$\psi(\psi^{-1}(I_i)) = I_i$ since $I_i \subset \psi(F)$. But $\psi(F_i)$ is irreducible by
proposition \ref{sommeimageirred}, so $\psi(F_i) = \mathring{I_i} = I_i$. It follows that
the $F_i$ are distinct since the $I_i$ are distinct.

Let us consider the chain $\{0\} = K_0 \subsetneq K_1\subsetneq \ldots \subsetneq K_r = F_0
\subsetneq F_1\subsetneq \ldots \subsetneq F_s$. We 
have to prove that it is maximal. The first
part of the chain is maximal by hypothesis. Now, let $G$ be irreducible such that
$F_i\subset G \subset F_{i+1}$ for $0\leq i \leq s-1$. 
It implies that $\psi(F_i)\subset \psi(G) \subset \psi(F_{i+1})$. 
But we have just seen that $\psi(F_i) = I_i$.
By maximality, we must have $\psi(G) = I_i$ or $\psi(G) = I_{i+1}$. 
If $G$ is an irreducible such that $F_s\subset G \subset F$, $I_s = \psi(F_s) \subset
\psi(G)$. But $I_s$ is the irreducible component of $\psi(F)$, hence $\psi(G) \subset I_s$.
So, in any case, we have $\psi(G) = I_i$ for $0\leq i\leq s$. We deduce that
$G+\ker \psi = \psi^{-1}(I_i)$.
Since $K_r = F_0 \subset G$, $G = G+K_r$, so $G$ has finite index in 
$G+\ker \psi = \psi^{-1}(I_i)$. It follows that $G$ is the irreducible component of
$\psi^{-1}(I_i)$, which is $F_i$ by definition.
\end{proof}

\section{Tangent space}

Let $f = \sum_{i=1}^n P_i(X_i)$ be an element of $\lan$. We define $d(f)$ to be the linear part 
of $f$. More precisely, $d(f) = \sum_{i=1}^n a_{0,i}X_i$ with $P_i = \sum_{j\geq 0}a_{j,i}X_i^{q^j}$.
We also have $d(f) = \sum_{i=1}^n \frac{\partial f}{\partial X_i}X_i$.
\begin{defi}\label{tanspa} Let $F\subset K^n$ be a $q$-variety. 
We define the tangent space of $F$, denoted 
by $T(F)$, by
$$T(F) = \bigcap_{f\in M(F)} \ker d(f) = \{(x_1,\ldots,x_n) \mid \forall f\in M(f),
\ d(f)(x_1,\ldots,x_n) = 0\}.$$
Note that $T(F)$ is a sub-$K$-vector space of $K^n$.
\end{defi}

\begin{proposition}\label{tanmap}
Let $F\subset K^n$ and $H\subset K^m$ be $q$-varieties, and  
$\psi : F \longrightarrow H$ be a morphism. Choose $f_1,\ldots,f_m\in \lan$ such that
for $x\in F$, $\psi(x) = (f_1(x),\ldots,f_m(x))$. Then the map $d(\psi)$ defined by
$$
\begin{array}{rcl}
d(\psi) : T(F) &\longrightarrow & T(H)\\
           x   & \mapsto        & (d(f_1)(x),\ldots,d(f_m)(x))
\end{array}
$$
is a well-defined morphism of $K$-vector spaces.
\end{proposition}

\begin{proof} If $g_1,\ldots,g_m\in \lan$ are such that
for $x\in F$, $\psi(x) = (g_1(x),\ldots,g_m(x))$. Then for $1\leq i \leq m$, $f_i-g_i\in M(F)$,
hence, by definition, for all $x\in T(F)$, $d(f_i-g_i)(x) = 0$, so $d(f_i)(x) =d(g_i)(x)$, 
proving that $d(\psi)$ does not depend on the choice of the $f_i$.

We still have to prove that $(d(f_1)(x),\ldots,d(f_m)(x))\in T(H)$ for $x\in T(F)$. 
Let $x \in T(F)$
and $g\in M(H)$. By chain rule, 
$$d(g)((d(f_1)(x),\ldots,d(f_m)(x))) = d(g(f_1,\ldots,f_m))(x).$$
But $g(f_1,\ldots,f_m)\in M(F)$ by construction, so
$d(g(f_1,\ldots,f_m))(x) = 0$ by definition
of $T(F)$. It follows that $d(g)((d(f_1)(x),\ldots,d(f_m)(x))) =0$, proving that
$(d(f_1)(x),\ldots,d(f_m)(x))\in T(H)$
\end{proof}

\begin{proposition}\label{tanfonc}
Let $F\subset K^n$and $H\subset K^m$ be $q$-varieties. The map
$$
\begin{array}{rcl}
d : \mor(F,H) &\longrightarrow & \Hom_{K}(T(F),T(H))\\
           \psi  & \mapsto        & d(\psi)
\end{array}
$$
is fonctorial. In particular, it implies that $T(F)$ depends only on the isomorphic
class of $F$.
\end{proposition}

\begin{proof}
This is nothing else but chain rule.
\end{proof}

\begin{proposition}\label{dimtan} Let $F$ be a $q$-variety, then $T(\mathring{F}) = T(F)$ and
$$\dim_K T(F) = \dim F.$$
\end{proposition}

\begin{proof}
We can suppose, up to an automorphism 
of $K^n$, that
$$M(F) = \kt P_1(X_1)+\ldots+\kt P_r(X_r)$$
for some $P_1,\ldots,P_r \in \kt\setminus \{0\}$, so
$F = \ker P_1\times\ldots\times\ker P_r\times K^{n-r}$ and 
$\mathring{F} = \{0\}\times\ldots\times\{0\}\times K^{n-r}$.
Since $M(F)$ is radical, $d(P_i) \not = 0$. It follows immediatly that 
$T(F) := \bigcap_{i=1}^r \ker d(P_i) = \{0\}\times\ldots\times\{0\}\times K^{n-r}$.
\end{proof}

\begin{proposition}\label{surjec}
Let $H\subset F$ be $q$-varieties, then
\begin{enumerate}
\item $T(H) \subset T(F)$.
\item The $K$-linear map $d(\Pi) : T(F) \longrightarrow T(F/H)$ is surjective.
\end{enumerate}
\end{proposition}

\begin{proof} The first property is obvious from the definition.

For the second one,  we suppose in a first time that $F$ is irreducible, hence, up to an isomorphism,
$F = K^n$. Now, we can also suppose that
$M(H) = \kt P_1(X_1)+\ldots+\kt P_r(X_r)$
for some $P_1,\ldots,P_r \in \kt\setminus \{0\}$. By construction, $F/H$ is the image of the
following morphism which is clearly surjective, hence $F/H = K^r$ : 
$$
\begin{array}{rcl}
\Pi : F = K^n & \longrightarrow & K^r \\
(x_1,\ldots,x_n) & \mapsto & (P_1(x_1),\ldots, P_r(X_r))
\end{array}
$$
By definition, the tangent map is given by
$$
\begin{array}{rcl}
d(\Pi) : T(F) = K^n & \longrightarrow & T(F/H) = K^r \\
(x_1,\ldots,x_n) & \mapsto & (d(P_1)(x_1),\ldots, d(P_r)(x_r))
\end{array}
$$
Since $M(H)$ is radical, $d(P_i) \not = 0$, proving that $d(\Pi)$ is surjective.

We return to the (quite technical) 
general case. Let $\mathring{F}$ be the irreducible component of $F$. Then
$\Pi(\mathring{F})$ is the irreducible component of $\Pi(F) = F/H$ by lemma \ref{imagecomp}. 
Suppose that we can show that $\Pi(\mathring{F}) = 
\mathring{F}/(\mathring{F}\cap H)$ as $q$-varieties.\footnote{they are obviously equal 
as $\Fq$-vector spaces.} By the previous case, 
we have a surjection
$T(\mathring{F}) \longrightarrow T(\Pi(\mathring{F})) = T(\mathring{F/H})$. But, by 
proposition \ref{dimtan}, $T(\mathring{F}) =T(F)$ and $T(\mathring{F/H})= T(F/H)$, proving the 
proposition.
\end{proof}

To finish the proof, we need the following lemma

\begin{lemma}
Let $H\subset F$ be $q$-varieties, $\Pi : F \longrightarrow F/H$ be the projection morphism
and $\mathring{F}$ be the 
irreducible component of $F$, then
$$\Pi(\mathring{F}) = \mathring{F}/(\mathring{F}\cap H).$$
\end{lemma}

\begin{proof} As usual, we can suppose that
$F = F_1\times\ldots\times F_r\times K^{n-r}$ with the $F_i$ being finite 
$\Fq$-vector spaces, so $\mathring{F} = \{0\}\times\ldots\times\{0\}\times K^{n-r}$.
Let $f_1,\ldots,f_m \in \Lambda_n$ be a generating set of $M(H)$. The map $\Pi$ is  defined 
by $\Pi(x) =(f_1(x),\ldots,f_m(x))$ (see proof of proposition \ref{quotient}). So 
\begin{multline*}
\Pi(\mathring{F}) =
\{(f_1(0,\ldots,0,x_{r+1},\ldots,x_n),\ldots,f_m(0,\ldots,0,x_{r+1},\ldots,x_n))\\
\mid (x_{r+1},\ldots,x_n) \in K^{n-r}\}.
\end{multline*}

Let us prove that $f_1,\ldots,f_m, X_1,\ldots,X_r$ is a generating set for\linebreak
$M(\mathring{F}\cap H)$. 

\begin{lemma}
Let $M\subset \lan$ be a radical module containing a separable polynomial $P_1(X_1)$
then  $M+\kt X_1$ is also radical.
\end{lemma}
\begin{proof}
Let us consider the $\kt$-modules canonical isomorphim
$$(M+\kt X_1)/M \simeq \kt X_1/(\kt X_1\cap M).$$
Since $\kt$ is euclidean, there exists $D_1\in \kt X_1$ such that $\kt X_1\cap M = \kt D_1$. 
By assumption, $P_1 \in \kt X_1\cap M$, hence $D_1$ divides $P_1$, so $D_1$ is separable. 
Furthermore, since $K$ is algebraically closed, $\tau\kt = \kt\tau$. This implies easily that
$\tau (\kt X_1/\kt D_1) = \kt X_1/\kt D_1$. Now let $P \in \lan$ such that $\tau P \in M+\kt X_1$.
By the previous result, there exists $Q\in M+\kt X_1$ such that $\tau P \equiv \tau Q \mod M$, hence
$\tau(P-Q) \in M$. But $M$ is radical, so $P-Q\in M$. It follows immediatly that $P \in  M+\kt X_1$ 
proving that $ M+\kt X_1$ is radical.
\end{proof}

By an obvious induction,  $M+\kt X_1+\ldots+\kt X_r$ is a radical module. Clearly, 
$Z(M+\kt X_1+\ldots+\kt X_r) = \mathring{F}\cap H$, proving that $M(\mathring{F}\cap H) =
M+\kt X_1+\ldots+\kt X_r$ and the claim.

It follows  that
\begin{align*}
\mathring{F}/(\mathring{F}\cap H) 
& = \{(f_1(x),\ldots,f_m(x),x_1,\ldots,x_r) \mid x\in \mathring{F}\}\\
& = \{(f_1(0,\ldots,0,x_{r+1},\ldots,x_n),\ldots,\\
&\qquad f_m(0,\ldots,0,x_{r+1},\ldots,x_n),0,\ldots,0)
\mid (x_{r+1},\ldots,x_n) \in K^{n-r}\}.
\end{align*}
This proves the lemma.
\end{proof}

We can now give a criteria for separable morphisms.

\begin{proposition}
Let $\psi : F \longrightarrow H$ be a morphism of $q$-varieties and  
$\overline{\psi}$ be the induced bijective morphism from $F/\ker \psi$ to $\psi(F)$. Then 
$\psi$ is separable if and only if $d(\overline{\psi})$ is a bijection.
\end{proposition}

\begin{remark} 
\begin{enumerate}
\item Since $\dim T(F/\ker \psi) = \dim F/\ker \psi = \dim \psi(F) = \dim T(\psi(F))$, 
$d(\overline{\psi})$ is a bijection if and only if it is injective or surjective.

\item Let us denote by $\widetilde{\psi}$ the induced morphism 
$\widetilde{\psi} :  F \longrightarrow \psi(F)$. It is clear that $\widetilde{\psi} 
= \overline{\psi}\circ \Pi$, so $d(\widetilde{\psi}) =  d(\overline{\psi})\circ d(\Pi)$.
By proposition \ref{surjec}, $d(\Pi)$ is surjective, hence $d(\overline{\psi})$ is surjective 
if and only if $d(\widetilde{\psi})$ is surjective.
\end{enumerate}
\end{remark}

\begin{proof}
Suppose that $\psi$ is separable. By definition, $\overline{\psi}$ is an isomorphism, hence
$d(\overline{\psi})$ is also an isomorphism by proposition \ref{tanfonc}.

Conversely, suppose that $d(\overline{\psi})$ is a bijection.  We assume first that $F$ is 
irreducible. Let $r = \dim F/\ker \psi = \dim \psi(F)$, so that, up to isomorphisms,
$F/\ker \psi = \psi(F) = K^r$ and $\overline{\psi} : K^r \longrightarrow K^r$ is a 
bijective morphism. Using lemma \ref{diag}, up to automorphisms, $\overline{\psi}$ is diagonal.
Since it is injective, the diagonal terms must be powers of $\tau$. But $d(\overline{\psi})$ 
is a bijection, hence the exponents must be $0$, so $\overline{\psi}$ is the identity map, 
up to isomorphisms.

We now consider the general case. By lemma \ref{imagecomp}, the image of $\mathring{F/\ker \psi}$ by 
$\overline{\psi}$ is the irreducible component of $\psi(F)$, so 
$\overline{\psi} : \mathring{F/\ker \psi} \longrightarrow  \mathring{\psi(F)}$ is a bijective 
morphism. Since $T(\mathring{H}) = T(H)$ for any $H$, $d(\overline{\psi}) : T(\mathring{F/\ker \psi})
\longrightarrow T(\mathring{\psi(F)}$ is a bijection by hypothesis. It follows from the previous case
that $\overline{\psi} : \mathring{F/\ker \psi} \longrightarrow  \mathring{\psi(F)}$ is an isomorphism.
We will be done if we can apply the following  lemma to the reciprocal map ${\overline{\psi}}^{-1}$.
\end{proof}

\begin{lemma}
Let $F$ and $H$ be $q$-varieties and $\psi : F \longrightarrow H$ be an $\Fq$-linear map
such that $\psi_{|\mathring{F}} : \mathring{F} \longrightarrow H$ is a morphism.
Then $\psi$ is a morphism.
\end{lemma}

\begin{proof} Without loss of generality, we can suppose that $H = K^m$ and 
$F = F_1\times\ldots F_r\times K^{n-r}$ with $F_i \subset K$  finite $\Fq$-vector spaces.
Let $f_1,\ldots,f_m$ be the functions defined by $\psi(x,0,\ldots,0) = (f_1(x),\ldots,f_m(x))$ for
$x\in F_1$. Using polynomial interpolation (see \cite{goss} chapter 1.3), there exists 
$P_1,\ldots,P_m\in \kt$ such that for all $x\in F_1$ and $1\leq i\leq m$, $f_i(x) = P_i(x)$. We set
$\psi_1(x) = (P_1(x),\ldots,P_m(x))$ for $x \in K$. By construction, for all $ x\in F_1$,
$\psi_1(x) = \psi(x,0,\ldots,0)$. The same way, we construct $\psi_2,\ldots, \psi_r$ and it is easy 
to check that for all $x \in F$, $\psi(x) = \psi_1(x_1)+\ldots+\psi_r(x_r)
+\psi(0,\ldots,0,x_{r+1},\ldots,x_n)$.
\end{proof}

\section{$A$-modules}

Let $A=\Fq[T]$ be the polynomial ring and $\delta : A \longrightarrow K$ be a morphism of 
$\Fq$-algebras. The kernel of $\delta$ is called the characteristic of $A$.

Let $F$ be a $q$-variety. The ring of endomorphism of $F$, $\mor(F,F)$, will
be denoted by $\End(F)$.

\begin{defi}\label{Amodule}
Let $F$ be a $q$-variety. We say that $(F,\Phi)$ is an $A$-module structure 
if $\Phi : A \longrightarrow \End(F)$ is a morphism of 
$\Fq$-algebras
such that, for all $a \in A$,
$$d(\Phi_a) = \delta(a)\Id_{T(F)}.$$
Let $(F, \Phi)$ and $(H,\Psi)$ be $A$-modules. 
We say that $U : F \longrightarrow G$ is an $A$-morphism if
it is a morphism of $q$-varieties and $A$-modules, i.e., for all $a\in A$ and for all $x\in F$,
$$U(\Phi_a(x))= \Psi_a(U(x)).$$
\end{defi}

\begin{remark}
In \cite{anderson}, the condition on $d(\Phi_a)$ is slightly different : $d(\Phi_a) = \delta(a)\Id_{T(F)}+N$
with $N$ an nilpotent endomorphism of $T(F)$. In the present article, $N$ is supposed to be zero
for simplicity 
but most properties should remain valid with $N \not = 0$.
\end{remark}

\begin{example} Let
$K$ be the algebraic closure of $\Fq(T)$, so that $\delta : A \longrightarrow K$ is just the 
inclusion. To define an $A$-module $(F, \Phi)$, it is sufficient to give $\Phi_T$.
\begin{enumerate}
\item The Carlitz module : we take $F = K$ and $\Phi_T = TX+X^q = T\tau^0+\tau$. It is the 
simplest non trivial $A$-module in dimension $1$. It is denoted by $C$. Let us denote
$C^-$ the $A$-module defined by $C^-_T = TX-X^q = T\tau^0-\tau$. These two $A$-modules 
are indeed isomorphic : let $\lambda \in K$ be such that $\lambda^{q-1} = -1$ and 
$U : K \longrightarrow K$ defined by $U(x) = \lambda x$. It is well-known and easy to 
check that $U$ is an isomorphism.

\item A Drinfeld module is an $A$-module with $F = K$ and $\Phi$ non trivial ($\Phi_a \not = \delta(a)\tau^0$).

\item Let $F = K^2$ and $\Phi$ be the $A$-module defined by
$$
\Phi_T = 
\begin{pmatrix}
T\tau^0  & \tau \\
\tau & T\tau^0
\end{pmatrix}.
$$
It means that $\Phi_T(x_1,x_2) = (Tx_1+x_2^q,x_1^q+Tx_2)$.
On the line $x_2 = x_1$, the $A$ action is given by
$\Phi_T(x,x) = (Tx+x^q,Tx+x^q) = (C_T(x),C_T(x))$, 
so the line $x_1 = x_2$ is an $A$-module and the induced
$A$-module structure is canonically isomophic the Carlitz module. 
The same is true on the line $x_2 = -x_1$: 
$\Phi_T(x,-x) = (C^-_T(x),-C^-_T(x))$. It follows that $\Phi$ is canonically 
isomorphic to the
direct sum of $C$ and $C^-$ if $p\not = 2$.
\end{enumerate}
\end{example}

\begin{proposition}
Let $(F,\Phi)$ be an $A$-module and $H\subset F$ be a $q$-variety. Then 
\begin{enumerate}
\item If, for all $a\in A$, $\Phi_a(H) \subset H$, then $H$ is an $A$-module.
\item If $H\subset F$ is an $A$-module, then $F/H$ is also an $A$-module. 
\item The irreducible component $\mathring{F}$ is an $A$-module.
\end{enumerate}
\end{proposition}

\begin{proof}
\begin{enumerate}
\item Since $T(H)\subset T(F)$ and by fonctoriality of the tangent map, we have
 $d({\Phi_a}_{|H}) = d({\Phi_a})_{|T(H)} = \delta(a)\Id_{T(H)}$, 
so $H$ is an $A$-module.
\item Let $\Pi : F \longrightarrow F/H$ be the projection map. Consider $\Pi\circ \Phi_a :
 F \longrightarrow F/H$. It is zero on $H$, hence by property of $F/H$ (see proposition
\ref{quotient}), there exists a
unique morphism $\Psi_a : F/H \longrightarrow F/H$ such that
$$ \Pi\circ \Phi_a = \Psi_a \circ \Pi .$$
By uniqueness of $\Psi_a$, it is clear that $a \longrightarrow \Psi_a$ is a ring morphism
from $A$ to $\End(F/H)$. Furthermore, taking the tangent maps, we get
$$ d(\Pi)\circ d(\Phi_a) = d(\Psi_a) \circ d(\Pi).$$
But $d(\Phi_a) = \delta(a)\Id_{T(F)}$ and $d(\Pi)$ is surjective by proposition \ref{surjec}.
It follows that $d(\Psi_a) = \delta(a)\Id_{T(F/H)}.$
\item Since the direct image of an irreducible is still irreducible, we have 
$\Phi_a(\mathring{F})\subset \mathring{F}$.
\end{enumerate}
\end{proof}

Let $(F,\Phi)$ be an $A$-module. Then $K\{F\}$ has an obvious $A$-module structure setting for 
$f\in K\{F\}$ and $a \in A$ :
$$a\cdot f = f\circ \Phi_a.$$
Furthermore,
the $A$ action commutes with the $K$ and the $\tau$ actions. In particular, $K\{F\}$ is a 
$K\otimes_{\Fq}A = K[T]$-module.

\begin{theorem}\label{rankphi} Let $(F,\Phi)$ be an $A$-module. Then the $K(T)$-vector space $K(F)$ defined
by $K(F) = K(T)\otimes_{K[T]}K\{F\}$ has finite dimension. Its dimension is called the rank of the module
$(F,\Phi)$ and is
denoted by~$r(F)$.
\end{theorem}

\begin{proof}
By definition, $F \subset K^n$ for some $n\in \N$. Hence $K\{F\}$ is a quotient of $\lan$ and is generated,
as a $\kt$-module, by the images of $\tau_1^0,\ldots,\tau_n^0$ in $K\{F\}$. These images are still denoted
$\tau_1^0,\ldots,\tau_n^0$ for simplicity.

Since $\kt$ is principal, $\lan$ and its quotient $K\{F\}$ are noetherian. It implies that there exists
$d\in \N$ such that $T^d.\tau_1^0$ belongs to the $\kt$-module generated by 
$T^{d-1}.\tau_1^0,\ldots,T^0.\tau_1^0$. It means that there exist $P_{d-1},\ldots,P_0 \in \kt$ such that
$$T^d.\tau_1^0 = \sum_{i=0}^{d-1} P_i\, T^i.\tau_1^0 = \sum_{i=0}^{d-1} T^i.P_i(\tau_1).$$
Rewriting this relation as a polynomial in $\tau_1$ with coefficients in $K[T]$, we get
\begin{equation}\label{relat}
\sum_{j=0}^s Q_j(T).\tau_1^j = 0
\end{equation}
for some $Q_j\in K[T]$ and $s\in \N$. Relation \eqref{relat} 
is not trivial because $Q_0$ is a monic polynomial of 
degree $d$, so we can suppose that $Q_s \not = 0$. It implies that $\tau_1^s$ belongs to the 
$K(T)$-vector space generated by $\tau_1^{s-1},\ldots,\tau_1^0$ in $K(F)$. Applying $\tau$ to relation 
\eqref{relat},
we get easily that $\tau_1^{s+1}$ belongs to the 
$K(T)$-vector space generated by $\tau_1^{s},\ldots,\tau_1^1$, hence to the 
$K(T)$-vector space generated by $\tau_1^{s-1},\ldots,\tau_1^0$. By induction, we get that all powers of
$\tau_1$ belongs to that vector space.

The same is obviously true for $\tau_2,\ldots,\tau_n$, proving the theorem.
\end{proof}

\begin{remark}
\begin{enumerate}
\item With G. Anderson definition (see \cite{anderson}), $K\{F\}$ is the motive associated to $F$. 
Furthermore, if  $K\{F\}$ is a free $K[T]$-module of rank $r$, 
it is clear that $\dim_{K(T)}K(F) = r$. Hence, our 
definition of the rank is coherent with Anderson's definition.
\item Let $F = K^n$ and $\Phi$ be the trivial module : for all $a\in A$ and $x\in K^n$, 
$$\Phi_a(x) = \delta(a)x.$$
In particular, for $1\leq i\leq n$, $(T-\delta(T)).\tau_i^0 = 0$. Composing with $\tau^m$, we get
$$(T-\delta(T)^{q^m}).\tau_i^m = 0.$$
It follows that $K\{F\}$ is a torsion module, hence $K(F) = 0$ and the rank of the trivial module is
$0$.
\end{enumerate}
\end{remark}

\section{Torsion points}

\begin{notation}
Let $(F,\Phi)$ be an $A$-module and $a\in A$. The $a$-torsion of $F$ will be denoted by
$\tor(a,F)$. In other words
$$\tor(a,F) = \{x\in F \mid \Phi_a(x) = 0\} = \ker \Phi_a.$$
It is an $\Fq$-vector space.
\end{notation}

\begin{theorem}\label{fintor}
Let $(F,\Phi)$ be an $A$-module and $a\in A\setminus \ker \delta$, then 
$\tor(a,F)$ is finite.
\end{theorem}

\begin{proof} By definition, $\tor(a,F)$ is the kernel of $\Phi_a : F \longrightarrow F$.
So the theorem is equivalent to $\dim \ker \Phi_a = 0$.  Now, by theorem \ref{rang},
$\dim \ker \Phi_a = \dim F - \dim \Phi_a(F)$, so we have to prove that 
$\dim \Phi_a(F) = \dim F$. Using proposition 
\ref{dimtan}, we have $\dim_K T(\Phi_a(F)) = \dim \Phi_a(F)$ and  
$\dim_K T(F) = \dim F$, hence it is
sufficient to prove $T(\Phi_a(F)) = T(F)$.

Since $ \Phi_a(F)\subset F$, we have  $T(\Phi_a(F)) \subset T(F)$ by proposition  \ref{surjec}.
Let us prove the reverse inclusion.
We consider the induced map $\widetilde{\Phi_a} : F \longrightarrow \Phi_a(F)$. So $\Phi_a = 
i\circ \widetilde{\Phi_a}$ where $i : \Phi_a(F) \longrightarrow F$ is the inclusion map.
Taking the tangent map, we get $\delta(a)\Id_{T(F)} =d(i)\circ d(\widetilde{\Phi_a})$.
But $d(i) : T(\Phi_a(F)) \longrightarrow T(F)$ is just the inclusion by 
proposition \ref{surjec}. 
It implies that $d(\widetilde{\Phi_a})(T(F)) = \delta(a)T(F) = T(F)$ 
since $\delta(a) \not = 0$. But $d(\widetilde{\Phi_a})(T(F))\subset T(\Phi_a(F))$, hence
$T(F) \subset T(\Phi_a(F))$.
\end{proof}

\begin{example} In the following examples, $\delta$ is supposed to be the inclusion map
and  $F = K^2$.
\begin{enumerate}
\item Let $\Phi$ be the $A$-module defined by
$$
\Phi_T = 
\begin{pmatrix}
T\tau^0  & \tau \\
\tau & T\tau^0
\end{pmatrix}.
$$
We have seen that $\Phi$ is isomorphic to the sum of two copies of the Carlitz module. 
It follows immediatly that for all $a \in A\setminus \{0\}$
$$\tor(a,F) = (A/aA)^2.$$
\item Let $\Phi$ be the $A$-module defined by
$$
\Phi_T = 
\begin{pmatrix}
T\tau^0  & \tau \\
0 & T\tau^0
\end{pmatrix}.
$$
One gets immediatly that for all $a\in A$, 
$\Phi_a = 
\begin{pmatrix}
a\tau^0  & P_a \\
0 & a\tau^0
\end{pmatrix}$ for some $P_a\in \kt$.
So $\tor(a,F) = \{0\}$ if $a\not = 0$.
\end{enumerate}
\end{example}

\begin{proposition}\label{formuledim}
Let $(F,\Phi)$ be an irreducible $A$-module and $a \in A\setminus \ker \delta$. 
Then
$$\dim_\Fq\tor(a,F) = \dim_K K\{F\}/a\cdot K\{F\}.$$
\end{proposition}

\begin{proof} Since $F$ is irreducible, we can suppose that $F = K^n$, so that $ K\{F\} = \lan$.
By lemma \ref{diag}, up to automorphisms, there exit $P_1(X_1),\ldots,P_r(X_r)\in \lan$ such that
$a\cdot\lan = \kt P_1(X_1)\oplus\ldots\oplus\kt P_r(X_r)$. It is clear that 
$$\tor(a,F) = Z(P_1(X_1),\ldots,P_r(X_r))= \ker P_1\times\ldots\times\ker P_r\times K^{n-r}.$$
Hence $r = n$ because $\tor(a,F)$ is finite.
Since $d(\Phi_a) = \delta(a)\Id$ with $\delta(a) \not = 0$, the $P_i$ must be separable. It follows
that 
$$\dim_\Fq\tor(a,F) = \sum_{i=1}^n\deg_\tau P_i = \dim_K \lan/a\cdot\lan.$$
\end{proof}

\begin{proposition}\label{tatemodule}
Let $(F,\Phi)$ be an irreducible $A$-module and $\pi \in A\setminus \ker \delta$
be a prime. Then there exists $r\in \N$ such that for all $n > 0$
$$\tor(\pi^n,F)= (A/\pi^nA)^r.$$
\end{proposition}

\begin{proof} Since $\dim \ker \Phi_\pi = 0$ by theorem \ref{fintor}, 
$\dim F = \dim \Phi_\pi(F)$. But $\Phi_\pi(F)\subset
F$ and $F$ is irreducible, so $\Phi_\pi(F) = F$, hence $\Phi_\pi$ is surjective.

By construction $\tor(\pi,F)$ is an $A/\pi A$-vector space which is finite by \ref{fintor}. 
Let $r$ be its dimension : $\tor(\pi,F)= (A/\pi A)^r$. Suppose that for some $n > 0$,
$\tor(\pi^n,F)= (A/\pi^nA)^r$. Using the elementary divisors theorem, there exists integers
$0< n_1 \leq n_2\leq \ldots \leq n_s\leq n+1$ such that
$$\tor(\pi^{n+1},F)= A/\pi^{n_1}A\times A/\pi^{n_2}A\times\ldots\times A/\pi^{n_s}A.$$
Considering $\tor(\pi,F)\subset \tor(\pi^{n+1},F)$, we get immediatly $s = r$. Furthermore,
the map $\Phi_\pi : \tor(\pi^{n+1},F)\longrightarrow \tor(\pi^{n},F)$ is clearly surjective 
with kernel equal to $\ker \Phi_\pi$. Hence $\Card \tor(\pi^{n+1},F) =
\Card\tor(\pi^{n},F)\times\Card \tor(\pi,F)$.
It implies that $n_1+n_2+\ldots+n_r = rn + r = r(n+1)$.
Since $n_i \leq n+1$, we must have $n_i = n+1$ for all $1\leq i\leq r$.
\end{proof}

\begin{example} In the following example, $\delta$ is supposed to be the inclusion map
and  $F = K^2$. Let $\Phi$ be the $A$-module defined by
$$
\Phi_T = 
\begin{pmatrix}
T\tau^0+\tau^2 & \tau \\
T\tau & T\tau^0
\end{pmatrix}.
$$
Let $\pi = T$. The elements of $\tor(\pi,F)$ are the solutions of
$$
\left\{
\begin{aligned}
Tx_1+x_1^{q^2}+x_2^q &= 0\\
Tx_1^q+Tx_2 &= 0
\end{aligned}
\right.
$$
The second equation implies that $x_2 = -x_1^q$, hence $x_2^q = -x_1^{q^2}$. Replacing $x_2^q$ by 
$-x_1^{q^2}$ in the first equation, we get $Tx_1 = 0$. It follows that $\tor(\pi,F)= \{0\}$ and,
by proposition \ref{tatemodule}, $\tor(\pi^n,F)= \{0\}$ for all $n > 0$.

Now let $\pi = T-1$.  The elements of $\tor(\pi,F)$ are the solutions of
$$
\left\{
\begin{aligned}
(T-1)x_1+x_1^{q^2}+x_2^q &= 0\\
Tx_1^q+(T-1)x_2 &= 0
\end{aligned}
\right.
$$
The second equation implies that $x_2 =-\frac{T}{T-1} x_1^q$, hence 
$x_2^q = -\frac{T^q}{T^q-1}x_1^{q^2}$. Replacing $x_2^q$ in the first equation, 
we get $Tx_1+(1-\frac{T^q}{T^q-1})x_1^{q^2} = 0$. It follows that $\dim_\Fq\tor(\pi,F)= 2$
and,
by proposition \ref{tatemodule}, $\tor(\pi^n,F) = (A/\pi^nA)^2$ for all $n > 0$.
\end{example}

We show now that $r$ is almost independant of $\pi$.

\begin{theorem}\label{rangconstant}
Let $(F,\Phi)$ be an $A$-module and $r(F)$ be its rank. Then there exists $c \in A\setminus \{0\}$ such that
for all $a\in A$, prime to $c$,
$$\tor(a,F) = (A/aA)^{r(F)}.$$
\end{theorem}

We start with two lemmas

\begin{lemma}\label{rangexact}
Let $(F,\Phi)$ be an $A$-module and $H\subset F$ be a submodule. Then
$$r(F) = r(H) +r(F/H).$$
\end{lemma}
\begin{proof}
By remark \ref{exactsequence}, we have an exact sequence of $\kt$-modules
$$0\rightarrow K\{F/H\} \rightarrow K\{F\}\rightarrow K\{H\}\rightarrow 0.$$
It is easy to check that is also a sequence of $K[T]$-modules. Since a localisation
is flat, we get an exact sequence of $K(T)$-vector spaces
$$0\rightarrow K(F/H) \rightarrow K(F)\rightarrow K(H)\rightarrow 0.$$
This proves the lemma.
\end{proof}

\begin{lemma}\label{torexact}
Let $(F,\Phi)$ be an $A$-module and $H\subset F$ be a submodule. Then there exists
$c \in A\setminus \{0\}$ such that
for all $a\in A$, prime to $c$, the following sequence is exact :
$$0\rightarrow \tor(a,H)\rightarrow \tor(a,F)\rightarrow \tor(a,F/H) \rightarrow 0.$$
\end{lemma}

\begin{proof}
The only non obvious part is that $\tor(a,F)\rightarrow \tor(a,F/H)$ is surjective. 
Let $\mathring{H}$ be the irreducible component of $H$. Since $H/\mathring{H}$ is finite, there exists
$c' \in A\setminus \{0\}$ such that $\Psi_{c'}(H/\mathring{H}) = 0$ where $\Psi$ is the induced $A$-module 
structure. It follows easily that for all $a\in A$,
prime to $c'$, $\Psi_a : H/\mathring{H} \rightarrow H/\mathring{H}$ is surjective.

Suppose that $a$ is also prime to $\ker \delta$. Hence 
$\Phi_a : \mathring{H} \rightarrow \mathring{H}$ is surjective 
(see proof of proposition \ref{tatemodule}). Let $y \in H$, then there exists $x\in H$ such that
$y \equiv \Phi_a(x) \mod \mathring{H}$. It means that $y-\Phi_a(x) \in \mathring{H}$. But there exists 
$z\in \mathring{H}$ such that $y-\Phi_a(x) = \Phi_a(z)$, hence $y = \Phi_a(x+z)$. It proves that
$\Phi_a : H \rightarrow H$ is surjective.

Let $\Pi : F \rightarrow F/H$ be the canonical surjection and $y\in F$ such that
$\Pi(y)\in \tor(a,F/H)$. By construction, $\Phi_a(y)\in H$. Since $\Phi_a: H \rightarrow H$
is surjective, there
exists $x\in H$ such that $\Phi_a(y) = \Phi_a(x)$. Hence $y-x\in \tor(a,F)$ and $\Pi(y-x) = \Pi(y)$.
This proves the lemma.
\end{proof}

\begin{proof}[Proof of Theorem \ref{rangconstant}]
Let $(F,\Phi)$ be an $A$-module and $\mathring{F}$ be its irreducible component. Since $F/\mathring{F}$
is finite, $K\{F/\mathring{F}\}$ has finite dimension over $K$. It implies that it is a $K[T]$ torsion
module, hence $r(F/\mathring{F}) = 0$. We then have $r(F) = r(\mathring{F})$ by lemma \ref{rangexact}.

Furthermore, there exist $c'\not = 0$ such that $\Psi_{c'}(F/\mathring{F}) = 0$. It implies that for all
$a\in A$ prime to $c'$, we have $\tor(a,F/\mathring{F}) = 0$. Let $c$ given by lemma \ref{torexact}, then,
for all $a\in A$ prime to $cc'$, $\tor(a,F) = \tor(a,\mathring{F})$.

So, without loss of generality, we can suppose that $F$ is irreducible, hence $F=K^n$ and 
$K\{F\} = \lan$. We can find $f_1,\ldots,f_{r(F)}\in K\{F\}$ such that there images in $K(F)$ form a basis.
Let $M \subset K\{F\}$ be the $K[T]$-module generated by  $f_1,\ldots,f_{r(F)}$. Since the images of 
 $f_1,\ldots,f_{r(F)}$ are linearly independant over $K(T)$, the  $f_1,\ldots,f_{r(F)}$ themselves are linearly
independant over $K[T]$. Hence $M$ is a free $K[T]$-module of rank $r(F)$.

Let $d \in \N$ strictly greater than the degrees of  $f_1,\ldots,f_{r(F)}$. Since the images 
of  $f_1,\ldots,f_{r(F)}$ form a basis of $K(F)$, for any $f\in K\{F\}$, one can find 
$P\in K[T]\setminus\{0\}$ such
that $Pf\in M$. So it is possible to find $P\in K[T]\setminus\{0\}$ such that for all $1\leq i\leq n$
and $j \leq d$,
$$P \tau_i^j \in M.$$
In particular, since $\tau(M)$ is included in the $K$-vector space generated by the 
$\tau_i^j,\ 1\leq i\leq n$
and $j \leq d$, we have $P\tau(M)\subset M$. It implies that for all $1\leq i\leq n$,
$$P\tau(P) \tau_i^{d+1} =P\tau(P\tau_i^d)\in P\tau(M) \subset M$$
where $\tau(\sum_{j=0}^s p_jT^j) = \sum_{j=0}^s \tau(p_j)T^j$. By an easy induction, we get that
for all $1\leq i\leq n$
and $j \in \N$,
$$P\tau(P)\tau^2(P)\ldots\tau^j(P)\tau_i^{d+j}\in M.$$
Let $a\in A \setminus \{0\}$ prime to $P$. Hence $\tau(a)$ is prime to $\tau(P)$ 
($\tau$ induces an automorphism of $K[T]$). But $\tau(a) = a$ since $a\in A = \Fq[T]$, hence
$a$ is prime to $P\tau(P)$.  By induction, 
we get that $a$ is prime to $P\tau(P)\tau^2(P)\ldots\tau^j(P)$ for any $j\in \N$.

The inclusion $M\subset \lan$ induces a morphism $M/aM \rightarrow \lan/a\lan$. We want to prove that it
is an isomorphism.

Let $f \in \lan$. Taking $j$ such that $d+j$ is greater than the degree of $f$, we have
$$P\tau(P)\tau^2(P)\ldots\tau^j(P)f \in M.$$
Since $a$ is prime to $P\tau(P)\tau^2(P)\ldots\tau^j(P)$, there exit $u,v \in K[T]$ satisfaying
$ua+vP\tau(P)\tau^2(P)\ldots\tau^j(P) = 1$, hence
\begin{align*}
f &= auf+vP\tau(P)\tau^2(P)\ldots\tau^j(P)f\\
& \equiv vP\tau(P)\tau^2(P)\ldots\tau^j(P)f \mod a\lan\\
&\in M \mod a\lan.
\end{align*}
It follows that the morphism is surjective.

Now, let $f\in M\cap a\lan$, so there exists $\lambda \in \lan$ such that $f = a\lambda$. As before,
taking $j$ such that $d+j$ is greater than the degree of $\lambda$, we have
$$P\tau(P)\tau^2(P)\ldots\tau^j(P)\lambda \in M.$$
Since $a$ is prime to $P\tau(P)\tau^2(P)\ldots\tau^j(P)$, there exit $u,v \in K[T]$ satisfaying
$ua+vP\tau(P)\tau^2(P)\ldots\tau^j(P) = 1$, hence
\begin{align*}
\lambda &= ua\lambda+vP\tau(P)\tau^2(P)\ldots\tau^j(P)\lambda\\
& =uf+vP\tau(P)\tau^2(P)\ldots\tau^j(P)\lambda
\in M.
\end{align*}
It follows that the morphism is injective.

So for all $a$ prime to $P$, we have $M/aM = \lan/a\lan$. Since $M$ is a free $K[T]$-module of rank $r(F)$,
$\dim_K \lan/a\lan = \dim_K M/aM = r(F)\deg_T a$. If $a$ is also prime to $\ker \delta$, proposition 
\ref{formuledim} implies that 
$$\dim_\Fq \tor(a,F) = r(F)\deg_T a.$$
Applying this formula in the special case $a = \pi$ a prime polynomial, we get
$$\dim_{A/\pi A} \tor(a,F) = r(F).$$
Now proposition \ref{tatemodule} says that for all $m >0$, 
$$\tor(\pi^m,F) = (A/\pi^m A)^{r(F)}.$$
We conclude the proof using chinese remainder theorem and 
$\tor(ab,F) =\tor(a,F)\times\tor(b,F)$ if $a$ and $b$ are coprime.
\end{proof}

\section{Jacobian}

Let $X \subset K^n$ be an affine algebraic curve. Roughly speaking, 
the Jacobian of $X$ is the smallest abelian variety containg $X$.
We want to define an analogue in our situation. In the classical case, we have 
the canonical action of $\Z$ on $K$ 
which induces a diagonal action on $K^n$. For $q$-varieties, we must choose the $A$-module structure.
This leads to the following definition.

\begin{defi}\label{jacob}
Let $(F,\Phi)$ be an $A$-module and $H\subset F$. Let $\jac_\Phi(H)$ be the 
intersection of all $A$-modules in $F$ containing $H$. It is clear that $\jac_\Phi(H)$
is an $A$-module and that it is the smallest $A$-module containing $H$. Note that
if $H$ is an irreducible $q$-variety then $\jac_\Phi(H)$ is also irreducible since the 
irreducible component of an $A$-module is an $A$-module.
\end{defi}

\begin{proposition}\label{picard}
Let $(F,\Phi)$ be an $A$-module 
and $H\subset F$ be an irreducible $q$-variety. Define the
Picard module associated to $H$ by 
$\pic(H) := A\otimes_\Fq H$. Then the canonical map
$$
\begin{array}{rcl}
\pic(H) & \longrightarrow & \jac_\phi(H) \\
 a\otimes x& \mapsto & \Phi_a(x)
\end{array}
$$
is surjective.
\end{proposition}

\begin{proof}
For $n\in \N$, define $H_n = H+\Phi_T(H)+\Phi_{T^2}(H)+\ldots+\Phi_{T^n}(H)$. 
Since the image and 
the sum of irreducibles are irreducible (see proposition \ref{sommeimageirred}), 
$H_n$ is irreducible. But the length of a chain of irreducibles is bounded by $\dim F$, so
there exists $n\in \N$ such that $H_{n+1} = H_n$.
It means that $\Phi_{T^{n+1}}(H)\subset 
H+\Phi_T(H)+\Phi_{T^2}(H)+\ldots+\Phi_{T^n}(H)$. It implies immediatly that
$H_n$ is stable by $\Phi_T$, hence $H_n$
is an $A$-module and it is easy to check
that any $A$-module containing $H$ must contain $H_n$, so $H_n = \jac_\Phi(H)$. This proves
the proposition.
\end{proof}

\begin{remark} The previous proposition
might not be true if $H$ is not supposed irreducible as shown in the following example. Let
$F= (K,\Phi)$ be an $A$-module, $x\in K$ not a torsion point and $H = \Fq x$. 
Then $\jac_\Phi(H) = F$ because it contains the free $A$-module of rank $1$ generated by $x$.
But this module, which is the
image of  $\pic(H)$, can not be equal to $F$ since $F$ has infinite 
rank by \cite{poonen}.
\end{remark}

\section{Some conjectures}

In \cite{denis}, L. Denis proposed three conjectures for $A$-modules of generic characteristic 
(i.e. $\ker \delta = \{0\}$). We give
an analogue of these conjectures. Indeed these analogues can be seen as special cases of Denis conjectures.

In the sequel, we suppose that $\delta : A \longrightarrow K$ is the inclusion map.

Let $(F,\Phi)$ be an $A$-module 
and $H\subset F$ be a $q$-variety. Let $x_1,\ldots,x_r \in F$ and $\Gamma = Ax_1+\ldots+Ax_r$ be the
module generated by $x_1,\ldots,x_r$ in~$F$.

The first conjecture is an analogue of Faltings theorem, see \cite{faltings}.

\begin{conjecture}\label{faltconj}
There exists $G \subset H$ an $A$-module such that
$G\cap \Gamma$ has finite index in $H\cap \Gamma$.
\end{conjecture}

This conjecture is obviously implied by the following one, which is an analogue of Mordell-Lang conjecture.

\begin{conjecture}\label{MLconj}
Let $\overline{\Gamma} = \{x \in F \mid \exists\ a\not = 0 \in A \text{ with }
\Phi_a(x) \in \Gamma\}$
There exists $G \subset H$ an $A$-module such that
$G\cap \overline{\Gamma}$ has finite index in $H\cap \overline{\Gamma}$.
\end{conjecture}

A special case of the previous conjecture is $\Gamma =  \{0\}$. It is an analogue of the
Manin-Mumford conjecture. In that case, $\overline{\Gamma}$ is just 
the set of all torsion points and is denoted by $\tor(F)$.

\begin{conjecture}\label{MMconj}
There exists $G \subset H$ an $A$-module such that
$G\cap \tor(F)$ has finite index in $H\cap \tor(F)$.
\end{conjecture}

The previous conjectures can be simplified using the following property.

\begin{proposition}
Let $(F,\Phi)$ be an $A$-module 
and $H\subset F$ be a $q$-variety. Then there exists an irreducible $A$-module $G_{\max} \subset H$ such that
for any irreducible $A$-module $G\subset H$, we have $G\subset G_{\max}$.
\end{proposition}

\begin{proof}
Let $G_0\subset H$ be an irreducible $A$-module with maximal dimension and 
$G\subset H$ be an irreducible $A$-module. Then $G_0+G$ is also an irreducible $A$-module by proposition
\ref{sommeimageirred}. By maximality of the dimension, $G_0+G=G_0$, hence $G\subset G_0$.
\end{proof}

As in \cite{ghioca}, we say that $H$ is sufficiently generic if $G_{\max} = \{0\}$. 
We now rewrite our conjectures with this extra condition.

Suppose that $H\subset F$ is a sufficiently generic $q$-variety. Then

\begin{conjecture}\label{genfaltconj}
$H\cap \Gamma$ is finite.
\end{conjecture}

\begin{conjecture}\label{genMLconj}
$H\cap \overline{\Gamma}$ is finite.
\end{conjecture}

\begin{conjecture}\label{genMMconj}
$H\cap \tor(F)$ is finite.
\end{conjecture}

\begin{proposition}
Conjectures \ref{faltconj}, \ref{MLconj} and \ref{MMconj} are equivalent, respectively, to
conjectures \ref{genfaltconj}, \ref{genMLconj} and \ref{genMMconj}
\end{proposition}

\begin{proof}[Proof of : Conjecture \ref{genMLconj} implies conjecture \ref{MLconj}]
Suppose that conjecture \ref{genMLconj} is true.
Let $H\subset F$ be a $q$-variety. It is clear that $H/G_{\max}$ is a sufficiently generic $q$-variety included
in the $A$-module $F/G_{\max}$. Let $\Pi : F \longrightarrow F/G_{\max}$ be the quotient map. Hence 
$H/G_{\max} = \Pi(H)$ and $F/G_{\max} = \Pi(F)$.
We apply conjecture \ref{genMLconj} to $\Pi(\Gamma)$  : $\Pi(H)\cap \overline{\Pi(\Gamma)}$ is finite.

Furthermore, let $y\in  \Pi(\overline{\Gamma})$, then there exists $x \in F$ and $a\not = 0\in A$ such that
$y = \Pi(x)$ and $\Phi_a(x)\in \Gamma$. It follows that $\Psi_a(y) = \Psi_a(\Pi(x)) = \Pi(\Phi_a(x))\in 
\Pi(\Gamma)$ where $\Psi$ is the $A$-module structure on $\Pi(F)$. Hence, $y\in \overline{\Pi(\Gamma)}$, so
$\Pi(\overline{\Gamma}) \subset \overline{\Pi(\Gamma)}$.

Now, $\Pi(H\cap \overline{\Gamma})\subset \Pi(H)\cap\Pi(\overline{\Gamma})\subset 
\Pi(H)\cap\overline{\Pi(\Gamma)}$. Hence $\Pi(H\cap \overline{\Gamma})$ is finite. 
Since $\ker \Pi = G_{\max}$, we conclude that $G_{\max}\cap\overline{\Gamma}$ 
has finite index in $H\cap \overline{\Gamma}$.
\end{proof}

Some cases of the conjectures are known. For examples, in \cite{ghioca}, D. Ghioca
proved that conjecture \ref{genMLconj}
holds when $F$ is a direct copy of a Drinfeld module and in \cite{scanlon}, T. Scanlon
proved that  conjecture \ref{MMconj}
holds with the same condition of $F$.

\end{document}